\documentclass{amsart}
\usepackage{amssymb}
\def\CC{\mathbb{C}}
\def\RR{\mathbb{R}}

\def\NN{\mathbb{N}}
\def\DD{\mathbb{D}}
\def\BB{\mathbb{B}}
\def\TT{\mathbb{T}}
\def\CDD{\overline{\DD}}

\def\su{\subset}
\def\OO{{\mathcal O}}
\def\cC{{\mathcal C}}

\def\eps{\varepsilon}
\DeclareMathOperator{\id}{id}

\DeclareMathOperator{\dist}{dist}
\DeclareMathOperator{\re}{Re}

\renewcommand{\phi}{\varphi}
\newtheorem{theorem}{Theorem}
\newtheorem{prop}{Proposition}

\newtheorem{lemma}[prop]{Lemma}
\theoremstyle{remark}\newtheorem{remark}[prop]{Remark}
\theoremstyle{remark}\newtheorem{example}[prop]{Example}

\begin{document}
\title{Comparison of invariant funcions and metrics}
\author{\L ukasz Kosi\'nski}
\address{Instytut Matematyki, Wydzia\l\ Matematyki i Informatyki, Uniwersytet Jagiello\'nski, ul. Prof. St. \L ojasiewicza 6, 30-348 Krak\'ow, Poland}
\email{lukasz.kosinski@gazeta.pl}
\begin{abstract}
It is shown that all invariant metrics and functions of bounded $\mathcal C^2$-smooth domain coincide on an open non-empty subset. The existence of Lempert-Burns-Krantz discs in $\mathcal C^2$-smooth domains is also discussed
\end{abstract}
\maketitle 

\section{Introduction}\label{1}
 The fundamental Lempert theorem states that $c_G=l_G$ and $\gamma_G=\kappa_G$ whenever $G$ is convex or smooth $\mathbb C-$convex domain in $\mathbb C^n$, where $c_G$ is the Carath\'eodory pseudodistance, $l_G$ denotes the Lempert function, $\gamma_G$ is the Carath\'eodory-Reiffen pseudometric and $\kappa_G$ is the Kobayashi-Royden pseudometric - see definitions in Section~\ref{def}.
 
It is well known that if $(d_G)$ is a contractible family of functions (respectively $(\delta)_G$ a contractible family of pseudometrics), where $G$ goes through the family of all domains in $\mathbb C^n$, then $c_G\leq d_G \leq l_G$ (resp. $\gamma_G\leq \delta_G \leq \kappa_G)$).
Therefore, the Lempert theorem may be formulated as follows: On any convex or smooth $\CC-$convex domain of $\mathbb C^n$ all invariant metrics are equal. 

This result is surprising as the functions and metrics mentioned above are holomorphic objects (and they are holomorphically invariant) and notions of convexity and $\CC-$convexity are just algebraic (topological) conditions.

There are some results describing properties of invariant metrics on strongly pseudoconvex domains. For example, $c_G$ and $k_G$ ($k_G$ denotes the Kobayashi pseudodistance, that is the biggest pseudodistance less or equal to $l_G$) are comparable in the sense that for any $\eps>0$ there is a compact subset $K=K(\eps)$ of a strongly pseudoconvex domain $G$ such that the inequality $$c_G(z,w)\leq k_G(z,w)\leq (1+\eps) c_G(z,w)$$ holds for $(z,w)\in G\times (G\setminus K)$ (see \cite{Ven1}). Another result in this direction is the following one due to M.~Abate (see \cite{Aba1}): $$\lim_{z\to \partial G}\frac{c_G(z_0,z)}{-\log\dist (z,\partial G)} = \lim_{z\to \partial G}\frac{k_G(z_0,z)}{-\log\dist (z,\partial G)}=1/2,$$ and the limits are locally uniform in the first variable $z_0\in G.$

Nevertheless, it is well known that invariant functions (metrics) are usually quite different even on smooth strongly pseudoconvex domains (and even when $n=1$) (see e.g. Section~\ref{examples}).

Our main result of this short note states that if $G$ is a $\mathcal C^2$-smooth bounded domain in $\CC^n$ (with no additional assumptions such a pseudoconvexity), where $n\geq 2$, then all invariant function (resp. metrics) are equal on some open subset of $G\times G$ (resp. $G\times \mathbb C^n$). The idea of the proof relies upon applying a counterpart of the Pinchuck scaling method and making use of stationary mappings which were used by Lempert in proving his deep result.

Then we show that modifying this method we may obtain the existence of Lempert-Burns-Krantz discs in $\mathcal C^2$ smooth domains (see \cite{Kim-Kra} for this definition). We end the paper with some technical lemmas (which follow simply from well known results and estimates).

\section{Notation and definitions}\label{def}

We start with some notation and we recall basic definitions appearing in the theory of invariant function and metrics. For a comprehensive monograph on this subject we refer the Reader to \cite{JP}.

Throughout the paper $\DD$ denotes the unit disc in the complex plain and $\mathbb T$ denotes the circle $\partial \DD$. By $\langle z ,w\rangle $ we denote the complex inner product on $\mathbb C^n$ and $z\bullet w$ is the dot product, i.e. $z\bullet w=\langle z,\bar w \rangle,$ $z,w\in \CC^n.$

Let $D\subset\CC^{n}$ be a domain and let $z,w\in D$, $v\in\CC^{n}$. The {\it Lempert function}\/ is defined as
\begin{equation}\label{lem}
l_{D}(z,w):=\inf\{p(0,\xi):\xi\in[0,1)\textnormal{ and }\exists f\in \mathcal{O}(\mathbb{D},D):f(0)=z,\ f(\xi)=w\}.
\end{equation} The {\it Kobayashi-Royden \emph{(}pseudo\emph{)}metric}\/ we define as
\begin{equation}\label{kob-roy}
\kappa_{D}(z;v):=\inf\{\lambda^{-1}:\lambda>0\text{ and }\exists f\in\mathcal{O}(\mathbb{D},D):f(0)=z,\ f'(0)=\lambda v\}.
\end{equation}

It is well known that generally $l_{D}$ does not satisfy a triangle inequality. Therefore, it is natural to consider the so-called \textit{Kobayashi \emph{(}pseudo\emph{)}distance} given by the formula \begin{multline*}k_{D}(w,z):=\sup\{d_{D}(w,z):(d_{D})\text{ is a family of holomorphically invariant} \\\text{pseudodistances less than or equal to }l_{D}\}.\end{multline*} Clearly, $$k_{D}(z,w)=\inf\left\{\sum_{j=1}^{N}l_{D}(z_{j-1},z_{j}):N\in\NN,\ z_{1},\ldots,z_{N}\in
D,\ z_{0}=z,\ z_{N}=w\right\}.$$

The next objects we are dealing with, are the \textit{Carath\'eodory \emph{(}pseudo\emph{)}distance}
$$c_{D}(z,w):=\sup\{p(F(z),F(w)):F\in\mathcal{O}(D,\DD)\}$$
and the \textit{Carath\'eodory-Reiffen \emph{(}pseudo\emph{)}metric}
$$\gamma_D(z;v):=\sup\{|F'(z)v|:F\in\mathcal{O}(D,\DD),\ F(z)=0\}.$$

Recall that a holomorphic mapping $f:\DD\longrightarrow D$ is said to be a \emph{complex geodesic} if for any $z,w\in f(\DD)$ there are $\zeta,\xi\in\DD$ such that $f(\zeta)=z$, $f(\xi)=w$ and $c_D(z,w)=p(\zeta,\xi)$ (resp. there are $\lambda_0\in \DD,$ $\alpha_0\in \CC$ such that $f(\lambda_0)= z,$ $X=\alpha_0 f'(\lambda_0)$ and $\gamma_D(z,X)=\gamma_D(\lambda_0, \alpha_0)$). 

Moreover, $f:\DD \to D$ is called an \emph{extremal} mapping if $k_D(z,w)=p(\zeta,\xi)$ for some $\zeta,\xi\in \DD$ such that $f(\zeta)=z$ and $f(\xi)=w$ (resp. there are $\lambda_0\in \DD,$ $\alpha_0\in \CC$ such that $f(\lambda_0)= z,$ $X=\alpha_0 f'(\lambda_0)$ and $\kappa_D(z,X)=\gamma_D(\lambda_0, \alpha_0)$).

Lempert introduced the concept of stationary map. This idea was originally derived by solving the Euler--Lagrange equations for the extremal problem. Let $D$ be a $\mathcal C^2$ smooth, bounded domain. Recall that  $f:\DD\longrightarrow D$ is a \emph{stationary mapping} if
\begin{enumerate}
\item[(1)] $f$ extends to a $\cC^{1/2}$-smooth mapping on $\CDD$;
\item[(2)] $f(\TT)\subset\partial D$;
\item[(3)] there exists a $\cC^{1/2}$-smooth function
$\rho:\TT\longrightarrow\RR_{>0}$ such that the mapping $\TT\ni\zeta\longmapsto\zeta
\rho(\zeta)\overline{\nu_D(f(\zeta))}\in\CC^{n}$ extends to a mapping $\widetilde{f}\in\OO(\DD)\cap\cC^{1/2}(\CDD)$ (we call $\tilde f$ a \emph{dual map} to $f$).
\end{enumerate}

Note that if $D$ is additionally convex and $f$ is a stationary mapping in $D$, then $\re \langle z-f(\zeta),\nu_D(f(\zeta))\rangle< 0$ for any $z\in D$ and $\zeta\in\TT$. Therefore, for any $z\in D$ the equation $(z-f(\cdot))\bullet \tilde f(\cdot)=0$ has exactly one solution $F(z)$ in $\DD.$ One may check that $F$ is a left inverse for $f,$ i.e. $F\circ f=\id.$

\section{Equality of invariant metrics and functions}
As mentioned in the Introduction, our main result is the following

\begin{theorem}\label{localization} Let $D\su\mathbb C^n$, $n\geq 2$, be a domain with $\mathcal C^2$ boundary. 

Then there is non-empty and open subset $U$ of $D\times D$ such that $$c_D(z,w)=l_D(z,w)\quad \text{for}\ (z,w)\in U.$$

Similarly, there is an non-empty and open subset $V$ of $D\times \mathbb C^n$ such that $$\kappa_D(z,X)=\gamma_D(z,X) \quad \text{for}\ (z,X)\in V.$$

\end{theorem}

Note that the assumption $n\geq 2$ is important. Actually, it is well known that $$c_A(z,w)< k_A(z,w),\quad (z,w)\in A\times A,$$ where $A$ is an annulus in the complex plane, that is $A:=\{z\in \mathbb C:\ r<|z|<1\},$ where $r<1.$

Similarly, the product property of the invariant metrics and pseudodistances ensures that the assumption about the smoothness is important, as well (consider the product of annuli).

We would like to point out that Theorem~\ref{localization} remains true if $C^2$-smoothness will be replaced with $C^{1,1}$-smoothness. The differences in the proofs between these two cases are only technical so for the simplicity we shall consider only $\mathcal C^2$ case.

\begin{lemma}\label{lem:loc} Let $D\su\mathbb C^n$, $n\geq 2$, be a domain. Assume that $a\in\partial D$ is such that $\partial D$ is $\mathcal C^2$ and strongly convex in a neighborhood of $a$. Then for any neighborhood $V_0$ of $a$ there is non-empty and open $U\subset (V_0\cap D)\times (V_0\cap D)$ with the following property:
 
 For any $(z,w)\in U$ there is a weak stationary mapping $f$ of $D\cap V_0$ passing through $(z,w)$ such that $f(\mathbb T)\su\partial D$.
\end{lemma}

\begin{proof}
Let $r$ be a $\mathcal C^2$ defining function in a neighborhood of $a$. The problem we are dealing with has a local character and $a$ is a point of strong convexity. Therefore, analyzing the Taylor series one may simple see that replacing $r$ with $r\circ\Psi$, where $\Psi$ is a local biholomorphism near $a$, we may assume that $a=(0,\ldots,0,1)$ and a defining function of $D$ near $a$ is of the form $r(z)=-1+||z||^2+h(z-a)$, where $h$ is $\mathcal C^2$ smooth in a neighborhood of $0$ and 
\begin{equation}\label{eq:f}h(z)=o(||z||^2),\quad \text{as}\ z\to 0.\end{equation} 
Of course, the similar holds for partial derivatives of $h$ of first order, i.e. 
\begin{equation}\label{eq:f'}D^{\alpha} h(z)=o(||z||),\ \text{as} \ z\to 0,\ \text{for any}\ \alpha\in \mathbb N^n,\ |\alpha|=1.\end{equation} In particular, $D^{\alpha}h(0)=0$ for any $\alpha\in \mathbb N^n_0$ such that $|\alpha|=2$.

Similarly as in \cite{Lem2} we consider the mappings\footnote{It should be noticed that here we use a version of the Pinchuck scaling method.}
$$A_t(z):=\left((1-t^2)^{1/2}\frac{z'}{1+tz_n},\frac{z_n+t}{1+tz_n}\right),\quad z=(z',z_n)\in\CC^{n-1}\times\DD,\,\,t\in(0,1).$$ Note that $A_t$ is an automorphism of $\BB_n$. Let \begin{equation}\label{rt} r_t(z):=\frac{|1+tz_n|^2}{1-t^2}r(A_t(z)),\quad t\in(0,1).\end{equation}

After some elementary calculations we infer that
\begin{multline}\label{eq:dok} r_t(z)=-1+||z||^2+\frac{|1+tz_n|^2}{1-t^2}h(A_t(z)-a)=\\ -1+||z||^2 + \frac{|1+tz_n|^2}{1-t^2} h\left((1-t^2)^{1/2} \frac{z'}{1+tz_n}, (1-t)\frac{z_n-1}{1+tz_n} \right). \end{multline}

Take any increasing sequence $\{t_{\mu}\}$ converging to $1$. Put $U_0:=\{z\in\mathbb C^n:\re z_n>-1/2\}$ and define $\rho(z):=-1+||z||^2$, $z\in \mathbb C^n$.

Observe that it follows from \eqref{eq:dok} that $ r_{t_{\mu}}|_{U_0}$ converges to $\rho|_{U_0}$ in $\mathcal C^2(U_0)$ topology. Actually, the local uniform convergence of functions $r_{t_{\mu}}$ follows simply from \eqref{eq:f}. Similarly, making use of \eqref{eq:f'} one may deduce the local uniform convergence of partial derivatives of the first order. The local uniform convergence of the partial derivatives of $r_{t_{\mu}}$ is a consequence of the continuity of the second-order partial derivatives of $h$.

Let $\chi$ be a $\mathcal C^{\infty}$ smooth function on $\mathbb R$ such that
\begin{itemize}
\item $\chi\equiv 0$ on $(-\infty,-1/2];$
\item $\chi\equiv 1$ on $[-1/4,\infty);$
\item $\chi$ is increasing.
\end{itemize}

Let us define \begin{equation}\label{rho} \tilde \rho_{\mu}(z):=\begin{cases} r_{t_{\mu}}(z),\quad  -1/2\leq \re z_n,\\
\chi(\re z_n) r_{\mu}(z)+ (1-\chi(\re z_n)) \rho(z),\quad -1/4\leq \re z_n<-1/2,\\
\rho(z),\quad \re z_n<-1/2. \end{cases}\end{equation}

Since $r_{t_{\mu}}$ converges to $\rho$ in $\mathcal C^2$ topology on $U_0$ we easily infer that $\tilde \rho_\nu$ converges to $\rho$ in $\mathcal C^2$ topology on $\mathbb C^n$. 

In particular, $(\tilde \rho_{\mu})$ converges locally uniformly to $\rho$. Let $(\eps_{\mu})$ be a sequence of positive numbers converging to $0$ such that $3\eps_{\mu+1} <\eps_\mu.$ There is a subsequence $(\rho_{s_\mu})$ such that $\sup_{\overline{\mathbb B}_n}|\rho_{s_\mu}-\rho|<\eps_\mu.$ Therefore, $\rho_\mu:=\tilde\rho_{s_\mu}+2\eps_\mu$ restricted to $\overline{\mathbb B}_n$ is strictly decreasing. 

Let $D_\mu$ be a connected component of $\{\rho_\mu<0\}$ containing $0$. Clearly $D_{\mu}$ are strongly convex provided that $\mu$ is big enough, as $\rho_nu$ converges to $\rho$ in $\mathcal C^2$ topology. Since $D_{\mu}$ increase to $\mathbb B$ we find that the sequence $l_{D_{\mu}}$ decreases $l_{\mathbb B_n}.$
Since $\rho_{\mu}$ converges to $\rho$ in $\mathcal C^2$ topology we easily deduce that there is a uniform $c>0$ such that any geodesic in $D_{\mu}$ such that $\dist(f(0),\partial D_{\nu})>1/c$ is $\mathcal C^{1/2}$ continuous and its $\mathcal C^{1/2}$ norm depends only on $\mu$ providing that $\mu$ is sufficiently large.

Now we proceed as follows. Let $V\subset \mathbb B_n\times \mathbb B_n$ be open and such that any geodesic passing through $(z,w)\in \overline V$ lies entirely in $U_1=\{\re z_n>-1/4\}.$ Let $W$ be non-empty, open and relatively compact in $W$.

It follows that for any $(z,w)\in W$ there is a geodesic $f_\mu$ in $D_{\mu}$ such that $f_\mu(0)=z$ and $f_\mu(\sigma_\mu)=w$ for some $\sigma_n>0.$ Passing, if necessary, to a subsequence we may assume that $f_\mu$ converges to a mapping $f_0:\mathbb D\to \overline{\mathbb B}_n$. Since $f_0(0)=z\in \mathbb B_n$ we get that $f_0(\mathbb D)\subset \mathbb B_n.$ Moreover, the statement of the Lempert theorem holds on $D_{\mu}$ that is $c_{D_{\nu}}=l_{D_{\nu}}$, hence we may easily see that $f_0$ is a complex geodesic $\mathbb B_n$ passing through $(z,w).$ Then uniqueness, uniform convergence and $\mathcal C^{1/2}$ uniform continuity implies that $f_\mu(\overline{\mathbb D})$ lies entirely in $\{\re z_n>-1/2\}$ providing that $\mu$ is big enough - see Lemma~\ref{c12} below.

Thus, a standard Baire argument implies the existence of an open non-empty subset $W$ of $V$ and a natural $\nu_1$ such that for any $(z,w)\in W$ and $\nu\geq \nu_1$ a geodesic of $D_n$ passing through $(z,w)$ (let us denote it by $f_{\nu,(z,w)}$) lies entirely in $W$. Actually, it suffices to apply the Baire's theorem to the family $\{G_\mu\},$ where $G_\mu:=\{(z,w)\in \bar V:\ f_{\nu,(z,w)}(\overline{\mathbb D})\subset \bar U_1,\ \nu\geq \mu\}$.

Observe that $g_{\nu}:=A_{t_{s_\nu}}\circ f_\nu$ is a stationary mapping of $D$. Since $g_{\nu}$ maps $\DD$ onto arbitrarily small neighborhood of $a$ provided that $\nu$ is sufficiently big, we immediately get the assertion.
\end{proof}

\begin{proof}[Proof of Theorem~\ref{localization}] Losing no generality let us assume that $0\in D.$ Fix a point $a$ in the topological boundary of $D$ whose distance from $0$ is the biggest. Then $a$ is a point of strict convexity of $D.$ Let $U'$ be an open and convex neighborhood of $a$ such that $D\cap U'$ is convex. Moreover \begin{equation}\re \langle z-a,\nu_D(a)\rangle<0\quad \text{for any}\ z\in D.\end{equation} Let $U''\subset \subset U'$ be a neighborhood of $a$ with the following property 

($\dag$) for any $\zeta\in \partial D\cap U''$ and any $z\in D\setminus U'$ one has the inequality $\re \langle z-\zeta, \nu_D(\zeta)\rangle \leq 0.$

Making use of Lemma~\ref{lem:loc} we get an open set $U$ in $D\times D$ such that for any $(z,w)\in U$ there is a weak stationary mapping of $D\cap U''$ passing through $(z,w)$ and entirely contained in $D\cap U''$. In particular, \begin{equation}\label{eq}\re \langle z-f(\eta),\nu_D(f(\eta))<0,\quad \zeta\in \TT\end{equation} for any $z\in U''$. Since $D\cap U'$ is convex we find that (\ref{eq}) holds for any $z\in U'$. Making use of ($\dag$) we infer that \eqref{eq} holds on the whole $D.$

From this we easily deduce that $f$ has a left inverse $F:D\to \mathbb D,$ hence $f$ is a complex geodesic (actually, $F(z)$ may be obtained as a unique solution of the equation $\langle z-f(\eta), \tilde f(\eta)\rangle=0$ with unknown $\eta\in \mathbb D$, where $\tilde f$ is a dual map of $f$).

\end{proof}

Let $D$ be a $\mathcal C^2$ smooth strongly pseudoconvex domain in $\mathbb C^n$ and let $a\in \partial D$. The theorem of Fornaess (see \cite{For}) gives a neighborhood $B$ of $a$, a strictly convex domain $C$ of $\mathbb C^n$ a mapping $\Phi: D\to \mathbb C^n$ extending holomorphically to a neighborhood of $D$ such that $\Phi(D)\subset C,$ $\Phi(B\setminus \overline D)\subset \mathbb C^n\setminus \overline C$, $\Phi^{-1}(\Phi(B))=B$ and the restriction $\Phi|_B:B \to \Phi(B)$ is biholomorphic. It follows from the reasoning presented above that we may construct a complex geodesic $f$ in $C$ lying entirely in $\Phi(B).$ Then $(\Phi|_B)^{-1}\circ f$ is a complex geodesic in $G$ ($F\circ \Phi$ is its left inverse, where $F$ is a left inverse of $f$ in $G$).

Using this standard reasoning we easily get the following:

\begin{theorem} Let $D$ be a $\mathcal C^2$-smooth strongly pseudoconvex domain in $\mathbb C^n$. Then for any $z_0\in \partial D$ and any neighborhood of $U$ of $z_0$ there is non-empty and open subset $V$ of $U\times U$ such that $c_D=\kappa_D$ on $V$.
\end{theorem}

\section{Lempert-Burns-Krantz discs in $\mathcal C^2$ smooth domains}

The main goal of this section is to apply the method presented above in order to show the existence of the so called Lempert-Burns-Krantz discs in $\mathcal C^2$-smooth strongly pseudoconvex domains. Recall that a Lempert-Burns-Krantz disc for points $a\in \partial D$ and $b\in D$, where $D$ is a domain of $\mathbb C^n$, is a geodesic $f$ of $D$, continuous up to $\bar{\mathbb D}$ such that $f(1)=a$ and $b\in f(\mathbb D).$ 

More precisely, we shall show the following

\begin{theorem}\label{lbk} Let $\Omega$ be a $\mathcal C^2$-smoothly bounded strongly pseudoconvex in $\mathbb C^n.$ Fix $p\in \partial \Omega.$ Then there is a non-empty and open subset $V$ of $\Omega$ such that for any $q\in V$ there exists a $f:\mathbb D\to \Omega$ such that
\begin{itemize}
\item $f$ is a complex geodesic in $\Omega,$
\item $f\in\mathcal C^{1/2}(\mathbb D)$ (in particular, $f$ extends continuously to $\bar {\mathbb D}$),
\item $q\in f(\mathbb D)$ and $f(1)=p.$
\end{itemize}

Moreover, $V$ may be arbitrarily close to $p.$

\end{theorem}

\begin{remark}
 Note that using the argument similar to the one used in the proof of Theorem~\ref{localization} one may show that any $\mathcal C^2$ smooth domain admits Lempert-Burns-Krantz discs.
\end{remark}

In the case when the domain $D$ was $\mathcal C^6$ smooth Theorem~\ref{lbk} was proved by Lempert in~\cite{Lem2} and formulated in the form above in \cite{BK}. It should be noted that the Lempert's method may be modified so that it works in the case of $\mathcal C^{2+\epsilon}$-smooth strictly pseudoconvex domains (see \cite{Lem1} for details). However it cannot be applied in $\mathcal C^2$-smooth case (more precisely, the crucial step of Lempert's arguments relied upon the implicit function theorem to the mapping which is not differentiable assuming only $\mathcal C^2-$smoothness).

The proof presented here is just a modification of the argument used in Section~\ref{1}. Note that we cannot use here a Baire-type argument and more subtle reasoning is necessary (we shall make use of estimates which are postponed to Section~\ref{3}).

\begin{proof}
As mentioned above the proof is just a slight modification of the proof of Theorem~\ref{localization} so we shall follow it. So we may assume that $r(z)=-1+||z||^2+h(z-a),$ where $a=(0,\ldots,0,1)\in \mathbb C^n,$ and $h$ is $\mathcal C^2$ smooth in a neighborhood of $0$, $h(z)=o(||z||^2),$ as $z\to 0.$
Let $r_t$ be given by \eqref{rt} and, similarly as in \eqref{rho}, put \begin{equation}\rho_{t}(z):=\begin{cases} r_{t}(z),\quad  -1/2\leq \re z_n,\\
\chi(\re z_n) r_{t}(z)+ (1-\chi(\re z_n)) \rho(z),\quad -1/4\leq \re z_n<-1/2,\\
\rho(z),\quad \re z_n<-1/2, \end{cases}\end{equation} where $\rho(z)=-1+||z||^2,$ $z\in \mathbb C^n.$ First observe that $D_t:=\{z\in \mathbb C^n:\ \rho_t(z)<0\}$ is connected, strongly pseudoconvex and $a$ lies in its boundary provided that $t$ is close enough to $1.$ Take any open sets $U$ and $V$ such that $U$ is relatively compact in $\mathbb B_n$, $a\in V$ and any geodesic of $\mathbb B_n$ passing through points $z\in U$ and $w\in V$ lies entirely in $\{\re z_n>0\}.$ Let $t$ be big enough (how big enough will be defined later). Fix $z\in U$ and take any sequence $(a_{\nu})\subset D_t$ converging to $a.$ Since $D_t$ is strictly convex we get that there is a complex geodesic $f_{\nu}$ of $D_t$ such that $f_{\nu}(0)=z$ and $f_{\nu}(\alpha_\nu)=a_\nu$ for some $\alpha_\nu\in (0,1)$ ($U\subset \subset D_t$ whenever $t$ is close to $1$). Clearly $\alpha_\nu$ tends to $1$. Since $f_{\nu}$ are uniformly $\mathcal C^{1/2}$-continuous we may find a subsequence of $f_{\nu}$ converging to a complex geodesic $f_z\in 
\mathcal C^{1/2}(\mathbb D)$ of $D_t$ such that $f_z(0)=z$ and $f_z(1)=a$ (see also \cite{Chl}, where the similar argument was used). Now, it follows from Lemma~\ref{close} that any complex geodesic $g$ of $\mathbb B_n$ for the pair $(z, f_z'(0))$, is close (in the sup-norm) to $f_z$. In particular, $g(1)\in V$ (if $t$ is big enough). Clearly $g$ is a complex geodesic for $z\in U$ and $g(1)\in V.$ Thus the image $g(\mathbb D)$ lies in $\{\re z_n>0\}.$ Therefore $f_z$ lies entirely in $\{\re z_n>-1/2\}$ (provided that $t$ is big enough).

So fixing $t$ sufficiently close to $1$ we easily verify that $\tilde U:=A_t(U)$ may be arbitrarily close to $a$ (as $A_t(\cdot)\to a$ uniformly on compact subsets of $\{\re z_n>-1\}$). Since stationary mappings are invariant under biholomorphisms we infer that $g_z:=A_t \circ f_z$ is an E-mapping in $D$ passing through $A_t(z)$ and $a$, whose image is contained in $\tilde U$. Since $U$ may be arbitrary small we may apply \cite{For}. Since E-mappings are geodesics on convex domains we easily find that $g_z$ is a complex geodesic.
\end{proof}

\begin{remark}
 It is well known that a geodesic $f$ of $\mathcal C^2$-smooth bounded convex domain is $\mathcal C^{\alpha}$ for any $\alpha<1.$ Note that if we knew that there is a left inverse for $f$ of class $\mathcal C^{\alpha}$ we would be able to formulate Burns-Krantz theorem for $D.$
 
 Of course, if $\tilde f$ were of $\mathcal C^1$ class, where $\tilde f$ is the dual map to $f$, then making use of the implicit function theorem and the equality $f'\bullet \tilde f=1$, we would easily find a $\mathcal C^1$-smooth left inverse. Note that such a claim was stated in \cite{Kra}, however we do not understand its proof.
\end{remark}

\section{Examples}\label{examples}
As mentioned above we have the following
\begin{example}
 Let $0\leq r_-<r_+<\infty$ and $A:=\{z\in \mathbb C: r_-<|z|<r_+\}/$. Then $$c_A(z,w)<l_A(z,w)=k_A(z,w)$$ for any $(z,w)\in A\times A.$
 
 Similarly, making use of the product property of invariant metrics we get that $$c_{A^n}(z,w)<l_{A^n}(z,w)=k_{A^n}(z,w),\quad z,w\in A^n.$$
\end{example}

\begin{example}
Let $D_{\alpha}=\{(z_1,z_2)\in \mathbb C^2:\ |z_1|<1,\ |z_2|<1,\ |z_1z_2|<\alpha\},$ where $\alpha\in (0,1].$ It follows from \cite{Klis} that any extremal mapping of $D_{\alpha}$ intersects one of the axes. Therefore, putting $D:=D_{\alpha}\cap \mathbb C^2_*$ we get $$c_D(z,w)<l_D(z,w)=k_D(z,w),\quad (z,w)\in D^2.$$
\end{example}

Of course, taking a sequence $(D_\nu)$ of strongly pseudoconvex domains exhausting $D_{\alpha}$ we see that $c_{D_\nu}\neq k_{D_\nu}$.

The author does not know whether the similar property holds for complete Reinhardt domains different then $D_{\alpha},$ $\alpha\in (0,1].$ Note that it may be shown (see Theorem... in \cite{Edi-Kos-Zwo}) that in any Reinhardt domain not euqal to $D_{\alpha}$ for any $\alpha\in (0,1]$ there is an extremal mapping omitting the axes.

Following this direction, note that if $D$ is a complete Reinhardt domain smooth near $z_0\in \partial D$ and $z_0$ is a point of strict pseudoconvexity, then it follows from our results and Fornaess' Theorem that the equality $c_D=l_D$ holds on some open, non-empty subset of $D^2.$

\section{Technical Lemmas}\label{3}

Recall that a complex geodesic of a strictly pseudoconvex domain is $\mathcal C^{1/2}$ continuous and speaking very generally its $\mathcal C^{1/2}$ depends only on the curvatures of the topological boundary of a domain, its diameter and the distance between $f(0)$ and $\partial D.$ See \cite{Lem2} for details (the results presented there are formulated for extremals in strictly pseudoconvex domains but their proofs work for complex geodesics in arbitrary strictly pseudoconvex domains - see also a detailed discussion on the Lempert's paper - \cite{KW}).

\begin{lemma}\label{c12} 
Assume that $D\subset \subset \mathbb B_n$ is a bounded $\mathcal C^2$ smooth strongly pseudoconvex domain of $\mathbb C^n$, $0\in D$. Let $r\in \mathcal C^2(\mathbb B_n)$ be a defining function of $D$ such that $\mathcal L r(a,X)\geq \alpha||X||^2,$ $a\in \partial D,$ $X\in \mathbb C^n$ for some $\alpha>0.$ Let $(r_\mu)\subset \mathcal C^2(\mathbb B_n)$ be a sequence converging to $r$ in a $\mathcal C^2$ topology on $\mathbb B_n$. By $D_\mu$ we denote a connected component of $\{z\in \mathbb B_n:\ r_\mu(z)<0\}$ containing the origin. Fix a compact subset $K$ of $D$.

Then every $l_{D_\mu}$-geodesic $f$ such that $f(0)\in K$ is $\mathcal C^{1/2}$ continuous and its $\mathcal C^{1/2}$ may be estimated by a constant independent of $\mu$ for $\mu$ big enough.
\end{lemma}

\begin{proof}[Sketch of proof of Lemma~\ref{c12}]
Is in \cite{KW} we have very easily shown that there is $c>0$ such that $D$ and $D\mu$ are in $\mathcal D(c)$ for $\mu>>1,$ where $\mathcal D(c)$ is the family defined in \cite{Lem1}. Thus it suffices to observe that Propositions~7 and 8 of \cite{Lem1} work for $\mathcal C^2$ smooth domains when we replace the assumption of being an $E$-mapping by the assumption of being a geodesic (the proofs given there may be taken over verbatim).
\end{proof}

\begin{remark}
It is worth of mentioning that using the argument of \cite{Chl}, which also works in a $\mathcal C^2$-smooth domain strictly pseudoconvex domain and keeping the notation from lemma above one may get that any $l_{D_\mu}$-geodesic such that has up to a composition with a M\"obius map $\mathcal C^{1/4}$ norm bounded by a constant independent of $\mu.$
\end{remark}

Note that the assumptions of the lemma presented above imply that $\{r_\mu<0\}$ is connected for $\mu>>1$. Moreover, using some simple calculus one may show that $\delta_{D_\mu}$ converges to $\delta_D$ in $\mathcal C^2$ topology, where $\delta_G$ denotes the signed distance to the bounded domain $G\subset \mathbb C^n,$ i.e. $$\delta_G(z):=\begin{cases}-\dist(z, \partial G),\quad \text{if}\ z\in \overline G,\\ \dist(z, \partial G),\quad \text{if}\ z\not\in \overline G.\end{cases} $$

\begin{remark}
We would like to recall once again that the lemmas above may also formulated and proved in $\mathcal C^{1,1}$ case (then the condition on curvatures should be naturally replaced by exterior and interior ball condition - see \cite{KW} for these ideas.
\end{remark}

If $D$ is a bounded $\mathcal C^2$ smooth strictly convex domain then for every $z\in D$ and $X\in \mathbb S_{n-1}:=\partial \mathbb B_n$ there is a unique $\mathcal C^{\alpha}-$smooth, $\alpha<0$, geodesic in $D$ for $(z,X)$, denoted by $f_{z,X}$ such that $f_{z,X}(0)=z$ and $f_{z,X}'(0)=\lambda_{z,X} X$ for some $\lambda_{z,X}>0.$

\begin{lemma}\label{close} Let $D$ be a $\mathcal C^2$ smooth strictly convex domain and let $r_0$ be its defining function given on a neighborhood $V$ of $\overline D.$ Let $K$ be a compact subset of $D.$ For $r\in \mathcal C^2(V)$ sufficiently close to $r_0$ in $\mathcal C^2$ topology on $V$ let $D_r$ denote a connected component of $\{x\in V:\ r(x)<0\}$ containing $0$ (note that $D^r$ is strictly convex providing that $r$ is sufficiently close to $r_0$). Let $f^r_{z,X}$ and $f_{z,X}$ be complex geodesic of $D_r$ and $D$ respectively such that $f_{z,X}^r(0)=f_{z,X}(0)=z$, $(f_{z,X}^r)'(0)=\lambda_{z,X}^r X$ and $f_{z,X}'(0)=\lambda_{z,X} X$ for some $\lambda_{z,X}^r, \lambda_{z,X}>0,$ where $z\in K$ and $X\in\mathbb S_{n-1}.$

Then $||f_{z,X}^r-f_{z,X}||_{\infty}\to 0$ as $r\to r_0$ in $\mathcal C^2$ topology on $V.$

\end{lemma}

\begin{proof} This is just a simply consequence of the uniqueness of complex geodesics in $D,$ Lemma~\ref{c12} and a standard compactness argument.

\end{proof}

\end{document}